\documentclass[10pt]{article}

\usepackage{graphicx}%
\usepackage{amsmath,amssymb,amsthm,upref,bm}%

\newtheorem{corollary}{Corollary}%
\newtheorem{theorem}{Theorem}%
\newtheorem{remark}{Remark}

\begin{document}

\baselineskip=4.4mm

\makeatletter

\newcommand{\E}{\mathrm{e}\kern0.2pt} 
\newcommand{\D}{\mathrm{d}\kern0.2pt}
\newcommand{\RR}{\mathbb{R}}
\newcommand{\CC}{\mathbb{C}}%
\newcommand{\ii}{\kern0.05em\mathrm{i}\kern0.05em}

\renewcommand{\Re}{\mathrm{Re}} 
\renewcommand{\Im}{\mathrm{Im}}

\def\bottomfraction{0.9}

\title{\bf Characterizations of discs via weighted means}

\author{Nikolay Kuznetsov}

\date{}

\maketitle

\vspace{-6mm}

\begin{center}
Laboratory for Mathematical Modelling of Wave Phenomena, \\ Institute for Problems
in Mechanical Engineering, Russian Academy of Sciences, \\ V.O., Bol'shoy pr. 61, St
Petersburg 199178, Russian Federation \\ E-mail: nikolay.g.kuznetsov@gmail.com
\end{center}

\begin{abstract}
\noindent New theorems characterizing analytically discs in the Euclidean plane
$\RR^2$ are proved. Weighted mean value properties of solutions to the modified
Helmholtz equation and harmonic functions are used for this purpose. The presence of
a logarithmic weight diminish coefficients in the mean value identities. A weighted
mean is also valid for solutions of the Helmholtz equation.
\end{abstract}

\setcounter{equation}{0}

\section{Introduction}

In this note, we consider a weighted mean value property of real-valued solutions to
the two-dimensional modified Helmholtz equation
\begin{equation}
\nabla^2 v - \mu^2 v = 0 , \quad \mu \in \RR \setminus \{0\} ;
\label{MHh}
\end{equation}
$\nabla = (\partial_1, \partial_2)$ is the gradient operator, $\partial_i = \partial
/ \partial x_i$. This property is used for a new analytic characterization of discs
in the Euclidean plane $\RR^2$.

Unfortunately, it is not commonly known that solutions of \eqref{MHh} are called
panharmonic (or $\mu$-panharmonic) functions by analogy with harmonic functions
solving the Laplace equation. This convenient abbreviation coined by Duffin \cite{D}
will be used in what follows.

In the extensive survey article \cite{NV}, the authors treat various mean value
properties of harmonic and caloric functions, saying nothing about panharmonic ones.
Meanwhile, Duffin~\cite{D} derived the mean value identity over circumferences in
$\RR^2$ for solutions of \eqref{MHh} in 1970. His result is closely related to the
obtained in this note, and so it is formulated below, but before that we introduce
some notation.

Let $x = (x_1, x_2)$ be a point in $\RR^2$, by $D_r (x) = \{ y \in \RR^2 : |y-x| < r
\}$ we denote the open disc of radius $r$ centred at $x$. The disc is called
admissible with respect to a domain $\Omega \subset \RR^2$ provided $\overline{D_r
(x)} \subset \Omega$, whereas $\partial D_r (x) = S_r (x)$ is called the admissible
circumference in this case. If $\Omega$ has a finite Lebesgue measure and a function
$f$ is integrable over $\Omega$ (continuous on $\Omega$), then
\[ M (f, \Omega) = \frac{1}{|\Omega|} \int_{\Omega} f (x) \, \D x \quad
\left( M (f, S_r (x)) = \frac{1}{2 \pi r} \int_{S_r (x)} f (y) \, \D S_y \right)
\]
is its area mean value over $\Omega$ (its mean value over an admissible
circumference, respectively); here $|\Omega|$ is the area of $\Omega$.

Now, we are in a position to recall some results related to those obtained in this
note. We begin with mean value properties analogous to those valid for harmonic
functions; the latter were reviewed in \cite{NV}.

\begin{theorem}[Duffin \cite{D}, Kuznetsov \cite{Ku}]
Let $\Omega$ be a domain in $\RR^2$. If $v$ is panharmonic in~$\Omega$, then
\begin{equation}
M (v, S_r (x)) = a^\circ (\mu r) \, v (x) \ \ \mbox{and} \ \ M (v, D_r (x)) =
a^\bullet (\mu r) \, v (x) \label{MM'}
\end{equation}
for every admissible disc $D_r (x);$ here $a^\circ (t) = I_0 (t)$ and $a^\bullet
(t) = 2 \, t^{-1} I_1 (t);$ $I_\nu$ denotes the modified Bessel function of order
$\nu$.
\end{theorem}

A new proof of the first identity \eqref{MM'} (originally due to Duffin) was
obtained in \cite{Ku} in the $m$-dimensional ($m \geq 2$) setting, whereas the
second one was derived in \cite{Ku} for the first time.

The aim of this note is twofold: (i) to consider weighted mean value properties of a
panharmonic and harmonic functions, and to compare them with those without weight;
(ii)~to prove inverse properties (the term coined in \cite{HN1} became widely
accepted) of the weighted mean value identities. One of them is similar the
following theorem, whose $m$-dimensional version was obtained recently.

\begin{theorem}[Kuznetsov \cite{Ku1}]
Let $\Omega \subset \RR^2$ be a bounded domain, and let $r > 0$ be such that $\pi
r^2 = |\Omega|$. If for some $\mu > 0$ and a point $x_0 \in \Omega$ the identity $v
(x_0) \, a^\bullet (\mu r) = M (v, \Omega)$ holds for every positive function $v$
satisfying equation \eqref{MHh} in $\Omega_r = \Omega \cup \left[ \cup_{x \in
\partial \Omega} D_r (x) \right]$, then $\Omega = D_r (x_0)$.
\end{theorem}

\section{Weighted mean value property}

The standard proof of mean value properties for harmonic functions usually starts
with the identity
\begin{equation*}
2 \pi w (x) = \int_{D_r (x)} \nabla^2 w (y) \log |x-y| \, \D y + \int_{S_r (x)}
\left[ w (y) \frac{\partial \log |x-y|}{\partial n_y} - \frac{\partial w}{\partial
n_y} \log |x-y| \right] \D S_y \, . \label{G3}
\end{equation*}
for an admissible disc $D_r (x)$. Along with the mean value properties, it also
implies
\begin{equation}
w (x) = \frac{1}{2 \pi r} \int_{S_r (x)} w (y) \, \D S_y - \frac{1}{2 \pi} \int_{D_r
(x)} \nabla^2 w (y) \log \frac{r}{|x-y|} \, \D y 
\label{Ev}
\end{equation}
This yields the following weighted version of the second identity \eqref{MM'}.

\begin{theorem}
Let $\Omega$ be a domain in $\RR^2$. If $v$ is panharmonic in $\Omega$, then
\begin{equation}
a (\mu r) \, v (x) = \frac{1}{\pi r^2} \int_{D_r (x)} v (y) \log \frac{r}{|x-y|} \,
\D y \, , \quad a (t) = \frac{2 \, [I_0 (t) - 1]}{t^2} \, ,
\label{MW'}
\end{equation}
for every admissible disc $D_r (x)$.
\end{theorem}

\begin{proof}
Substituting $v$ into \eqref{Ev} and taking into account the first identity
\eqref{MM'} and equation \eqref{MHh} in the first and second terms, respectively, on
the right-hand side, we obtain
\[ 2 v (x) = 2 a^\circ (\mu r) \, v (x)- \frac{(\mu r)^2}{\pi r^2} \int_{D_r
(x)} v (y) \log \frac{r}{|x-y|} \, \D y \, ,
\]
after multiplying by two both sides. Now \eqref{MW'} follows by rearranging.
\end{proof}

The behaviour of the logarithmic weight in the area mean value identity \eqref{MW'}
is quite simple: it is a positive function of $y$ within $D_r (x)$, growing from
zero attained for $y \in S_r (x)$ to infinity as $|x-y| \to 0$, and is negative for
$y \notin \overline{D_r (x)}$.

In view of the behaviour of $I_0$, one obtains that
\[ a (0) = \lim_{t \to +0} a (t) = 1/2 \, ,
\]
and $a (t)$ increases monotonically from this value to infinity similar to
$a^\bullet (t)$. Moreover,
\[ a^\bullet (t) - a (t) = 2 \left[ t I_1 (t) - I_0 (t) + 1 \right] / t^2 > 0 \quad
\mbox{for all} \ t \in [0, \infty) .
\]
An immediate consequence of \eqref{MW'} and the fact that $a (t) > 1/2$ for $t > 0$
is the following.

\begin{corollary}
Let $\Omega$ be a domain in $\RR^2$, and let $v$ be a $\mu$-panharmonic in $\Omega$
for some $\mu > 0$. If $v \geq 0$ does not vanish identically in $\Omega$, then
\begin{equation}
\frac{1}{2} \, v (x) < \frac{1}{\pi r^2} \int_{D_r (x)} v (y) \log \frac{r}{|x-y|}
\, \D y \label{ineq}
\end{equation}
for every admissible disc $D_r (x)$.
\end{corollary}

\begin{remark}
{\rm In the limit $\mu \to +0$, one obtains the Laplace equation from \eqref{MHh},
whereas identity \eqref{MW'} turns into
\begin{equation}
\frac{1}{2} \, v (x) = \frac{1}{\pi r^2} \int_{D_r (x)} v (y) \log \frac{r}{|x-y|}
\, \D y \, , \label{har}
\end{equation}
where $x$ is a point of a bounded domain $\Omega$ and $D_r (x)$ is an admissible
disc. Therefore, it is reasonable to conjecture that \eqref{har} constitutes a
weighted mean value identity for a harmonic function $v$. To the best author's
knowledge, this identity has not been proven yet.}
\end{remark}

\begin{remark}
{\rm According to the first identity \eqref{MM'}, every nonnegative,
$\mu$-panharmonic function is subharmonic. In view of Corollary~1 and Remark~1, one
might expect that inequality \eqref{ineq} holds for nonnegative subharmonic
functions which does not vanish identically.}
\end{remark}

\section{Characterizations of discs}

The following analogue of Theorem 2 is based on weighted means of positive
panharmonic functions.

\begin{theorem}
Let $\Omega \subset \RR^2$ be a bounded domain, and let $r >0$ be such that
$|\Omega| \geq \pi r^2$. If for a point $x_0 \in \Omega$ and some $\mu > 0$ the
weighted mean value identity
\begin{equation}
a (\mu r) \, v (x_0) = \frac{1}{|\Omega|} \int_{\Omega} v (y) \log
\frac{r}{|x_0 - y|} \, \D y \label{MW''}
\end{equation}
holds for every positive function $v$ satisfying equation \eqref{MHh} in $\Omega_r =
\Omega \cup \left[ \cup_{x \in \partial \Omega} D_r (x) \right]$, then $\Omega =
D_r (x_0)$.
\end{theorem}

Prior to proving Theorem 4, we notice that the radially symmetric function
\begin{equation*}
V (x) = I_0 (\mu |x|) \, , \quad x \in \RR^2 ,
\label{U}
\end{equation*}
monotonically increases from one to infinity as $|x|$ goes from zero to infinity.
Also, it solves equation \eqref{MHh} in $\RR^2$; indeed, the Poisson's integral for
$I_0$ (see \cite{NU}, p. 223) yields that
\begin{equation*}
V (x) = \frac{2}{\pi} \int_0^1 \frac{\cosh (\mu |x| s)}{(1 - s^2)^{1/2}} \, \D s \,
, \label{PU}
\end{equation*}
which is easy to differentiate, thus verifying \eqref{MHh}.

\begin{proof}[Proof of Theorem 4.]
Without loss of generality, we suppose that the domain $\Omega$ is located so that
$x_0$ coincides with the origin. Let us show that the assumption $\Omega \neq D_r
(0)$ leads to a contradiction. For this purpose we consider bounded open sets $G_i
= \Omega \setminus \overline{D_r (0)}$ (nonempty by the assumption about $\Omega$
and $r$) and $G_e = D_r (0) \setminus \overline{\Omega}$ (possibly empty).

Taking into account that $V (0) = 1$, we write \eqref{MW''} for $V$ as follows:
\begin{equation}
|\Omega| \, a (\mu r) = \int_{\Omega} V (y) \log \frac{r}{|y|} \, \D y \, ,
\label{1}
\end{equation}
Since identity \eqref{MM'} holds for $V$ over $D_r (0)$, we write it in the same
way:
\begin{equation}
\pi r^2 a (\mu r) = \int_{D_r (0)} V (y) \log \frac{r}{|y|} \, \D y \, . \label{2}
\end{equation}
Subtracting \eqref{2} from \eqref{1}, we obtain
\begin{equation*}
\left[ |\Omega| - \pi r^2 \right] a (\mu r) = \int_{G_i} V (y) \log \frac{r}{|y|} \,
\D y - \int_{G_e} V (y) \log \frac{r}{|y|} \, \D y \, .
\end{equation*}
Here the difference on the right-hand side is negative. Indeed, $V > 0$ everywhere,
whereas $\log (r / |y|) < 0$ on $G_i \neq \emptyset$, because $|y| > r$ there.
Hence, the first term is negative. If $G_e \neq \emptyset$, then the second integral
is positive because $\log (r / |y|) > 0$ on $G_e$, where $|y| < r$. On the other
hand, the expression on the left-hand side is nonnegative. The obtained
contradiction proves the theorem.
\end{proof}

\begin{remark}
{\rm Comparing Theorems 2 and 4, we observe two points worth mentioning.

First, Theorem 4 is essentially two-dimensional, whereas the $m$-dimensional $(m
\geq~2)$ version of Theorem 2 is proved in \cite{Ku1}.

Second, the common feature of both theorems is that their proofs involve the
function $V$.}
\end{remark}

It occurs that discs are characterized in the same way via harmonic functions
provided identity \eqref{har} is true for them.

\begin{theorem}
Let $\Omega \subset \RR^2$ be a bounded domain, and let $r >0$ be such that
$|\Omega| \geq \pi r^2$. Suppose that identity \eqref{har} is true for functions
harmonic in $\Omega_r$. If for a point $x_0 \in \Omega$ the weighted mean value
identity
\begin{equation*}
\frac{1}{2} v (x_0) = \frac{1}{|\Omega|} \int_{\Omega} v (y) \log
\frac{r}{|x_0 - y|} \, \D y \label{MW''}
\end{equation*}
holds for every positive function $v$ harmonic in $\Omega_r$, then $\Omega = D_r
(x_0)$.
\end{theorem}

To prove this assertion one has to repeat literally the proof of Theorem~4, but
using $H (x) \equiv 1$ instead of $V (x)$. Identity \eqref{har} is valid for $H (x)$
as one readily finds by a direct calculation.

\section{Weighted mean of solutions to the Helmholtz equation}

For real-valued solutions of the Helmholtz equation
\begin{equation}
\nabla^2 u + \lambda^2 u = 0 , \quad \lambda \in \RR \setminus \{0\} ,
\label{Hh}
\end{equation}
the mean value identities analogous to those in Theorem 1 are valid. Of course,
$I_0$ and $I_1$ must be replaced by the Bessel functions $J_0$ and $J_1$,
respectively, in the formulae for $a^\circ$ and $a^\bullet$; see~\cite{Ku}, pp.~675
and~677. However, the inverse mean value property analogous to Theorem~2 is more
complicated for solutions of \eqref{Hh}, because a restriction on the domain's size
is imposed; see \cite{Ku2}, Remark~2.1. It is required since the function $U (x) =
J_0 (\lambda |x|)$, used in the proof instead of $V (x)$, is monotonic only for
$\lambda |x| \in (0, j_{1,1})$; here $j_{1,1}$ is the first positive zero of $J_1$.

Therefore, it is interesting to find out how the logarithmic weight changes the mean
value identity for discs in this case, and whether it allows to improve the inverse
property obtained in \cite{Ku2}. It is clear that minor changes in the proof of
Theorem 3 yield the following.

\begin{theorem}
Let $\Omega$ be a domain in $\RR^2$. If $u$ is a solution of \eqref{Hh} in
$\Omega$, then
\begin{equation}
\tilde a (\lambda r) \, u (x) = \frac{1}{\pi r^2} \int_{D_r (x)} u (y) \log
\frac{r}{|x-y|} \, \D y \, , \quad \tilde a (t) = \frac{2 \, [1 - J_0 (t)]}{t^2} \,
, \label{New}
\end{equation}
for every admissible disc $D_r (x)$.
\end{theorem}

The behaviour of $\tilde a$ is as follows: $\tilde a (0) = \lim_{t \to +0} \tilde a
(t) = 1/2$, whereas $\tilde a (t)$ asymptotes zero as $t \to +\infty$ decreasing
nonmonotonically, but remaining positive. The latter property of $\tilde a (t)$
distinguishes it from $2 \, t^{-1} J_1 (t)$---the coefficient in the identity for
discs with the mean value without weight. The latter coefficient has infinitely many
zeros.

\begin{remark}
{\rm As in Remark 1, the Laplace equation results from \eqref{Hh} in the limit
$\lambda \to +0$, whereas identity \eqref{New} turns into
\begin{equation*}
\frac{1}{2} \, u (x) = \frac{1}{\pi r^2} \int_{D_r (x)} u (y) \log \frac{r}{|x-y|}
\, \D y \, , \label{harm}
\end{equation*}
thus confirming the conjecture made in Remark 1 that this equality constitutes a
weighted mean value identity for harmonic functions.}
\end{remark}

Thus, the role of weight is essential, However, it is easy to establish that the
restriction on the size of domain, under which the inverse theorem analogous to that
obtained in \cite{Ku2} is valid, is even stronger when identity \eqref{New} is used,
and so it has no advantage.

\end{document}